\documentclass[letterpaper, 10 pt, conference]{ieeeconf}  
\IEEEoverridecommandlockouts    
\usepackage{cite}
\newcommand{\eat}[1]{}
\usepackage{booktabs}
\usepackage{color}
\usepackage [autostyle, english = american]{csquotes}
\MakeOuterQuote{"}
\usepackage{xcolor}

\usepackage{multicol}

\usepackage{comment}
\usepackage{mathtools}
\usepackage{amsmath}
\usepackage{amstext}
\usepackage{amssymb}
\usepackage{amsfonts}
\usepackage{float}

\usepackage{amsthm}  
\newtheorem{theorem}{Theorem}

\newtheorem{definition}{Definition}

\newtheorem{remark}{Remark}

\newtheorem{proposition}{Proposition}

\usepackage{subfig}
\usepackage{caption}
\usepackage{todonotes}

\makeatletter

\newcommand{\Rmnum}[1]{\expandafter\@slowromancap\romannumeral #1@}
\usepackage{savesym}

\usepackage{algorithm}
\usepackage{algorithmicx} 
\usepackage{algpseudocode} %
\savesymbol{AND}
\usepackage[group-separator={,},group-minimum-digits={3}]{siunitx}

\usepackage{graphicx} 
\usepackage{epsfig} 

\usepackage{times} 
\usepackage{amsmath} 
\usepackage{amssymb}  
\makeatletter
\let\NAT@parse\undefined
\makeatother
\usepackage{hyperref}
\hypersetup{
   colorlinks=true,
    linkcolor= blue,
    allcolors=blue,
    citecolor = blue,
    filecolor=black,      
    urlcolor=blue,
    }
\usepackage{mathrsfs}
\title{\LARGE \bf
Combining Learning and Control in Linear Systems
}
\author{Andreas A. Malikopoulos, \emph{IEEE Senior Member}%
\thanks{This work was supported by NSF under Grants CNS-2149520 and CMMI-2219761.}%
\thanks{The author is with the School of Civil and Environmental Engineering, Cornell University, Ithaca, NY 14850, USA. (email: \tt\small{amaliko@cornell.edu}.)}
}

\begin{document}

\maketitle
\thispagestyle{empty}
\pagestyle{empty}

\begin{abstract}
In this paper, we provide a theoretical framework that separates the control and learning tasks in a linear system. This separation allows us to combine offline model-based control with online learning approaches and thus circumvent current challenges in deriving optimal control strategies in applications where a large volume of data is added to the system gradually in real time and not altogether in advance. We provide an analytical example to illustrate the framework.

\end{abstract}


\section{Introduction}
\PARstart{R}{einforcement} learning (RL) \cite{Sutton1998a,Bertsekas1996} has emerged as an adaptive method to control  systems  \cite{Sutton:1992ub} with unknown dynamics \cite{Kordabad:2023aa}.
There have also been research efforts on developing  learning approaches using Bayesian analysis to address such problems
\cite{Fisac:2019wb}.
Other approaches over the years have focused on direct or indirect RL methods including robust learning-based  \cite{Bouffard:2012wp,Aswani:2013ue}, learning-based model predictive control \cite{Zhang:2020wf,Rosolia:2018wv,Hewing:2020uv} on autonomous racing cars \cite{Rosolia:2020uo},  real-time learning  \cite{Malikopoulos2009,Malikopoulos2011} of powertain systems with respect to the driver's driving style \cite{Malikopoulos2010a,Malikopoulos2009a}, 
learning for traffic control  \cite{Wu2017FlowAA} for transferring optimal policies \cite{chalaki2020ICCA,jang2019simulation}, decentralized learning for stochastic games \cite{Arslan:2017vo}, learning for optimal social routing  \cite{Krichene:2018we} and congestion games \cite{Krichene:2015vx}, and learning for enhanced security against replay attacks in cyber-physical systems \cite{Sahoo:2020tx}.

The implications of the strategies derived using a model, which is typically different from the actual system, have been reported in \cite{Kara:2018vu}.  A recent paper \cite{Subramanian2020ApproximateIS} proposed approximate learning of an information state to address problems when the dynamics of the actual system are not known.  Other efforts have combined adaptive control with RL to derive control strategies in real time \cite{Guha2021OnlinePF}. Space constraints prevent us from discussing the complete list of papers reported in the literature in this area. Two survey papers \cite{Recht2018ATO,Kiumarsi:2018tq}, however, include a comprehensive review of the  RL approaches.

In some applications, we encounter a volume of data gradually incorporated into the system. To derive the optimal control strategy in such applications, we typically use a model \cite{Malikopoulos2016b}. However,  model-based control might not effectively facilitate optimal solutions partly due to the existing discrepancy between the model and the actual system. On the other hand, supervised learning approaches might not always facilitate robust solutions using training data derived offline. Similarly,  RL approaches might impose undesired implications on the system's robustness. 

In this paper, we investigate how to circumvent these challenges at the intersection of learning and control. 
We derive sufficient statistics that can represent the system's growing data. This sufficient statistics is called \textit{information state} of the system and takes values in a time-invariant space.  This information state can be used to derive \textit{separated control strategies}, which are related to the separation between the estimation of the information state and the derivation of the control strategy. Given this separation, for any control strategy at time $t$,  the evolution of the information state of the system does not depend on the control strategy at $t$ but only on the realization value of the control at  $t$ \cite{Kumar1986}. Thus, the evolution of the information states is separated from the choice of the current control strategy. 
Hence, the optimal control strategy is derived offline using the information state, which can be learned online using standard techniques  \cite{Gyorfi:2007aa,Brand:1999aa}  while data are incorporated into the system. 
This approach departs from traditional model-based and supervised (or unsupervised) learning approaches. The framework could effectively facilitate optimal solutions in a wide range of applications where a large volume of data is added to the system gradually in real time and not altogether in advance, such as emerging mobility systems, mobility markets, smart power grids, power systems, social media platforms, robot cooperation, and the Internet of Things.

The structure of the paper is organized as follows. In Section \ref{sec:2}, we present the formulation of the optimal control problem. In Section \ref{sec:3}, we introduce the separated control strategies. In Section \ref{sec:4}, we illustrate the framework with a simple analytical example. Finally, we draw concluding remarks in Section \ref{sec:5}.

\section{Problem Formulation}
\label{sec:2}
\subsection{Notation}
In our exposition, we denote by $\mathbb{E}[\cdot]$ the expectation of random variables,  by $\mathbb{P}(\cdot)$  the probability of an event, and by $p(\cdot)$ the probability density function. 
We denote by $\mathbb{E}^{\bf{g}}[\cdot]$, $\mathbb{P}^{\bf{g}}(\cdot)$, and $p^{\bf{g}}(\cdot)$  that the expectation, probability, and probability density function, respectively, depending on the choice of the control strategy $\bf{g}$. 
Random variables are denoted with upper case letters, and their realizations with lower case letters, e.g., for a random variable $X_t$, $x_t$ denotes its realization. In some occasions, we denote the expected value of a random variable with lower case letter, e.g., $\mathbb{E}[X_t]=x_t$. Subscripts denote time. The shorthand notation $X_{0:T}$ denotes the the vector of random variables $(X_0,\ldots, X_T)$, and the shorthand notation $x_{0:T}$  denotes the vector of their realization $(x_0,\ldots, x_T)$.

\subsection{Modeling Framework} 
\label{sec:2a}

We  consider a linear system  in which a volume of data is added to the system gradually and not altogether in advance. 
We aim to find sufficient statistics that can be used to compress the increasing data of the system without loss of optimality  \cite{Striebel1965}. These statistics are a conditional probability of the system's state at time $t\in\mathbb{R}_{\ge0}$ given all data available up until $t$, which is called the information state of the system. We use this information state to derive separated control strategies. By deriving separated control strategies, we can derive the optimal control strategy offline with respect to the information state and then use  learning methods  to learn the information state online.

\begin{figure}
	\centering
	\includegraphics[width=1\linewidth, keepaspectratio]{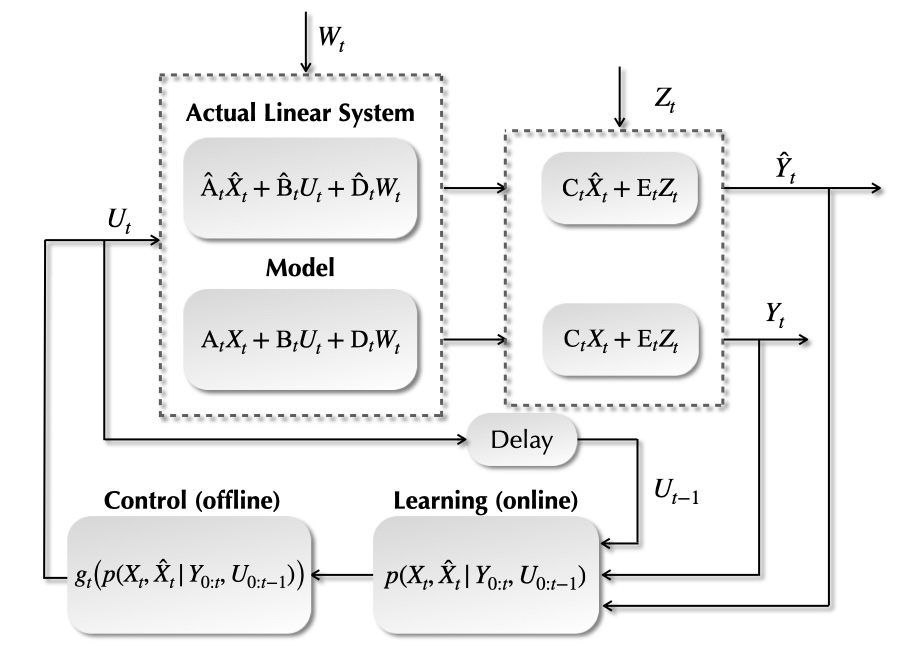} 
	\caption{The proposed framework on separating learning and control, where the separated control strategy is applied to both the actual system and the system's model  in parallel.}%
	\label{fig:2}%
\end{figure}

In particular, in our framework illustrated in Fig. \ref{fig:2}, we seek to use the actual linear system that we wish to optimally control online, in parallel with a model of the system that is available.  Let $X_t$, $t = 0,1,\ldots,T$, $T\in\mathbb{N},$ be a random variable that corresponds to the state  of the system's model and $\hat{X}_t$, $t = 0,1,\ldots,T$, be a random variable that corresponds to the state of the actual system. Both  $X_t$ and $\hat{X}_t$  are defined on an appropriate probability space and take values in $\mathbb{R}^n$, $n\in\mathbb{N}$. The control $U_t$ of the actual system is a random variable  defined on the same probability space and takes values in $\mathbb{R}^m$, $m\in\mathbb{N}$.  Given an initial state $X_0$, the model of the linear system is 
\begin{align}\label{eq:state}
	X_{t+1}=\textbf{A}_t  X_t +\textbf{B}_t U_t+\textbf{D}_t W_t, ~t = 0,1,\ldots,T-1,
\end{align}
where  $\textbf{A}_t, \textbf{B}_t,$ and $\textbf{D}_t$ are matrices of appropriate dimensions, and $W_t\in\mathbb{R}^r$, $r\in\mathbb{N}$, is a random variable that corresponds to the external, uncontrollable disturbance. 
Given the same initial state $X_0$, the actual system  is represented by 
\begin{align}\label{eq:statereal}
	\hat{X}_{t+1}=\hat{\textbf{A}}_t \hat{X}_t + \hat{\textbf{B}}_t U_t +\hat{\textbf{D}}_t W_t, ~t = 0,1,\ldots,T-1,
\end{align}
where $\hat{\textbf{A}}_t, \hat{\textbf{B}}_t,$ and $\hat{\textbf{D}}_t$ are matrices of appropriate dimensions.  The sequence $\{W_t; ~t = 0,1,\ldots,T-1\}$ is a sequence of independent random variables independent of the initial state $X_0$.

At the time $t = 0,1,\ldots,T-1$, we make an observation $Y_t\in\mathbb{R}^p$, $p\in\mathbb{N},$ of the model's output described by the observation equation 
\begin{align}\label{eq:observe}
	Y_t = \textbf{C}_t X_t +\textbf{E}_t Z_t, 
\end{align}
where $\textbf{C}_t, \textbf{E}_t$ are matrices of appropriate dimensions, and $Z_t\in\mathbb{R}^s$, $s\in\mathbb{N}$, is a random variable that represents the sensor's noise. Similarly, at time $t$, we make an observation $\hat{Y}_t\in\mathbb{R}^p$, $p\in\mathbb{N},$ of the actual system, described by the observation equation 
\begin{align}\label{eq:observereal}
	\hat{Y}_t = \textbf{C}_t \hat{X}_t +\textbf{E}_t Z_t,
\end{align}
Note  $\{Z_t\},$ ~$t=0,\ldots,T-1,$ is a sequence of independent random variables that are also independent of $\{W_t\}$,~$t=0,\ldots,T-1,$ and the initial state $X_0$. 

A control strategy $\textbf{g}=\{g_t\}$ of the system yields a decision 
\begin{align}\label{eq:control}
	U_t =g_t(\hat{Y}_{0:t}, U_{0:t-1}),~t=0,\ldots,T-1,
\end{align}
where  the measurable function $g_t$ is the control law. The feasible set of the control strategies is $\mathcal{G}$, i.e., $\textbf{g}\in\mathcal{G}$.

\textbf{Problem 1}~[Actual linear system]: \label{problem1}
Derive the optimal control strategy $\textbf{g}^*\in\mathcal{G}$ which minimizes the following cost of the actual system, 
\begin{align}\label{eq:cost}
	\hat{J}(\textbf{g})=\mathbb{E}^{\textbf{g}}\left[\sum_{t=0}^{T-1} c_t(\hat{X}_t, U_t)+c_T(\hat{X}_T)\right],
\end{align}
where the expectation in \eqref{eq:cost} is taken with respect to the joint probability distribution of $\hat{X}_t$ and  $U_t$ imposed by the control strategy $\textbf{g}\in\mathcal{G}$; $c_t(\cdot, \cdot): \mathcal{X}\times \mathcal{U}_t \to\mathbb{R}$ 
is the  cost function of the system at $t$, and $c_T(\cdot): \mathcal{X} \to\mathbb{R}$ is the  cost function at $T$. The probability distribution of the primitive random variables $X_0$,  $\{W_t\}$, $\{Z_t\}$, the cost functions $\{c_t(\cdot,\cdot)\}$ for $t=0,\ldots,T-1$ and $c_T(\cdot)$, and the matrices $\textbf{C}_t, \textbf{E}_t$ for $t=0,\ldots,T-1$ are all known. However, the matrices $\hat{\textbf{A}}_t, \hat{\textbf{B}}_t, \hat{\textbf{D}}_t$ are not known for $t=0,\ldots,T-1$.


\section{Separating Learning and Control  Tasks} \label{sec:3}
Let $\textbf{g}=\{g_t;~ t=0,\ldots,T-1\}$, $\textbf{g}\in\mathcal{G},$ be a control strategy which yields a decision $U_t =g_t(Y_{0:t}, U_{0:t-1})$ using the model of the linear system. 
We establish an  information state by using in parallel the system's model and the actual system as shown in Fig. \ref{fig:2}. 

The probability density function $p(X_{t}, \hat{X}_t ~|~Y_{0:t}, U_{0:t-1})$ is the information state (defined formally next), denoted by $\Pi_{t}(Y_{0:t}, U_{0:t-1})(X_{t},\hat{X}_t)$. To simplify notation, in what follows, the information state $\Pi_{t}(Y_{0:t}, U_{0:t-1})$ $(X_{t},\hat{X}_{t})$ at $t$ is denoted by $\Pi_t$ without the arguments, which will be used only if it is required in our exposition.

\begin{definition} \label{def:infoteam}
	The information state, $\Pi_{t}(Y_{0:t}, U_{0:t-1})$ $(X_{t},\hat{X}_{t})$,  is (a) a function of  $(Y_{0:t}, U_{0:t-1})$, and (b) its evolution $\Pi_{t+1}(Y_{0:t+1}, U_{0:t})(X_{t+1},\hat{X}_{t+1})$  at the next step $t+1$ depends on $\Pi_{t}(Y_{0:t}, U_{0:t-1})(X_{t},\hat{X}_{t})$, $Y_{t+1}$, and $U_{t}$.
\end{definition}

\begin{theorem}\label{theo:y_t}
	The information state $\Pi_{t}(Y_{0:t}, U_{0:t-1})(X_{t},\hat{X}_t)$ does not depend on the control strategy $\textbf{g}\in\mathcal{G}$.
	Furthermore, there exists a function $\phi_t$ such that
	\begin{align}\label{eq:xt1}
		&\Pi_{t+1}(Y_{0:t+1}, U_{0:t})(X_{t+1},\hat{X}_{t+1}) 
		\nonumber\\
		&= \phi_t\big[ \Pi_{t}(Y_{0:t}, U_{0:t-1})(X_{t},\hat{X}_t), Y_{t+1}, U_t \big],
	\end{align}
	for all $t=0,1,\ldots, T-1.$
\end{theorem}

The result of Theorem \ref{theo:y_t} is equivalent to the result of \cite[Theorem 2]{Malikopoulos2022a} when the system's information structure  is classical \cite{Yuksel2013,Schuppen2015} and the controller has perfect recall \cite{bertsekas1995dynamic,Kumar1986}.


\begin{definition}\label{def:septeam}
	A control strategy $\textbf{g}\in\mathcal{G}$, $\textbf{g}=\{g_t\}$,~$t=0,\ldots,T-1$, is called \textit{separated} if $g_t$ depends on $Y_{0:t}=(Y_0,\ldots, Y_{t})$ and  $U_{0:t-1}=(U_0,\ldots, U_{t-1})$  through the information state, i.e., $U_t  = g_t\big(\Pi_{t}(Y_{0:t}, U_{0:t-1})(X_{t},\hat{X}_{t})\big)$. Let $\mathcal{G}^s\subseteq\mathcal{G}$ denote the set of all separated control strategies.
\end{definition}

Since the dynamics of the actual system are not known, we cannot solve Problem $1$. Thus, to obtain the optimal strategy in Problem $1$, we formulate the following problem that we solve offline using the system's model \eqref{eq:state}.

\textbf{Problem 2:} \label{problem2}
Derive offline the optimal separated strategy $\textbf{g}^*\in\mathcal{G}^s$ to minimize the following cost function
\begin{align}			
	J&(\textbf{g};\hat{x}_{0:T})
	= \mathbb{E}^{\textbf{g}}\Bigg[\sum_{t=0}^{T-1}\Big[c_t(X_t, U_t)+ \beta \cdot|Y_{t+1}- \hat{Y}_{t+1}|^2\Big]\nonumber\\
	&+c_T(X_T) \Bigg], \nonumber
\end{align}	
or, using \eqref{eq:observe} and \eqref{eq:observereal}, 
\begin{align}
 J&(\textbf{g};\hat{x}_{0:T})
 = \mathbb{E}^{\textbf{g}}&\Bigg[\sum_{t=0}^{T-1}\Big[c_t(X_t, U_t)+ \beta \cdot|X_{t+1}- \hat{X}_{t+1}|^2\Big]\nonumber\\
	&+c_T(X_T) \Bigg], \label{eq:costreal}	
\end{align}	
where  $\beta$ adjusts the units of the norm accordingly, while the norm penalizes the discrepancy between the expected values of the state of the system's model and the state of the actual system. Since we solve \eqref{eq:costreal} offline using model \eqref{eq:state}, no information about the actual system is available, and thus the expected values $\hat{x}_{0:T}=(\hat{x}_0,\ldots, \hat{x}_T)$ of the states $\hat{X}_{0:T}=(\hat{X}_0,\ldots, \hat{X}_T)$ of the actual system are not known. Hence, when we derive the optimal control strategy $\textbf{g}^*$, it is parameterized with respect to all possible values $\hat{x}_{0:T}$. 

Next, to  obtain offline the optimal separated control strategy in Problem 2, we use the information state  $\Pi_{t}(Y_{0:t}, U_{0:t-1})(X_{t},\hat{X}_{t})$. 
It can be shown \cite{Malikopoulos2022a} that we can derive a classical dynamic program decomposition with respect to $\Pi_{t}$ to yield  a  separated control strategy, namely, a control strategy $\textbf{g}=\{g_t\}$, $t=0,\ldots,T-1$ where $g_t$ depends on $Y_{0:t+1}$ and  $U_{0:t}$ only through the information state, i.e., $U_t  = g_t\big(\Pi_{t}(Y_{0:t}, U_{0:t-1})(X_{t},\hat{X}_{t})\big)$. 

The separated control strategy is derived offline, thus, it is parameterized with respect to the potential expected values $\hat{x}_{0:T}$ of the state $\hat{X}_{t}$  of the actual system. Then, we apply the parameterized strategy to the actual system and the system's model  in parallel  (Fig. \ref{fig:2}), and we collect data from both. Using these data, we  compute  $\Pi_{t}(Y_{0:T}, U_{0:T-1})(X_{t+1},\hat{X}_{t+1})$  online.

\begin{proposition} \label{theo:CPSstate}
	The information state $\Pi_{t}(Y_{0:t}, U_{0:t-1})(X_{t},$ $\hat{X}_{t})$ of the system illustrated in Fig. \ref{fig:2} can be represented as a function of  $p(X_{t} ~|~ Y_{0:t}, U_{0:t-1})$, $p(\hat{X}_{t} ~|~ \hat{Y}_{0:t}, U_{0:t-1})$, and $p(\hat{Y}_{0:t}~|~U_{{0:t}-1})$.
\end{proposition}
\begin{proof}
	Recall 
	\begin{align}	
		\Pi_{t}(Y_{0:t}, U_{0:t-1})(X_{t},\hat{X}_t) =p(X_{t},\hat{X}_t~|~Y_{0:t}, U_{0:t-1}).
	\end{align}
	Next,
	\begin{align}	
		&p(X_{t},\hat{X}_t~|~Y_{0:t}, U_{0:t-1})\nonumber\\
		&=\frac{p(\hat{X}_t~|~ X_{t},Y_{0:t}, U_{0:t-1})\cdot p(X_t, Y_{0:t}, U_{0:t-1})}{p(Y_{0:t}, U_{0:t-1})} \nonumber\\
		&=\frac{p(\hat{X}_t~|~ U_{0:t-1})\cdot p(X_t, Y_{0:t}, U_{0:t-1})}{p(Y_{0:t}, U_{0:t-1})} \nonumber\\
		&= p(\hat{X}_t~|~ U_{0:t-1})\cdot p(X_{t} ~|~ Y_{0:t}, U_{0:t-1}).
		\label{theoCPSstate:1b}	
	\end{align}	
	In the second equality, we used the fact that $\hat{X}_t$ does not depend on $X_t$ and $Y_{0:t}$, and in the third equality, we applied Bayes' rule.  
	
	Next, we write the first term in \eqref{theoCPSstate:1b} as follows
	\begin{align}
		p(\hat{X}_t~|~ U_{0:t-1})=\sum_{\hat{Y}_{0:t}} p(\hat{X}_t~|~\hat{Y}_{0:t}, U_{0:t-1})\cdot p(\hat{Y}_{0:t}~|~ U_{0:t-1}).
		\label{theoCPSstate:1c}	
	\end{align}
	Substituting \eqref{theoCPSstate:1c} into \eqref{theoCPSstate:1b}, the result follows.
\end{proof}

\begin{remark} \label{rem:CPSstate}
	The conditional probability $p(X_{t}~|~Y_{0:t},$ $U_{0:t-1})$ can be obtained easily using the model offline. The conditional probability $p(\hat{X}_{t}~|~\hat{Y}_{0:t}, U_{0:t-1})$ can be obtained from the Kalman filter to estimate $\hat{X}_{t}$ first,  and then through recursive equations starting from the initial prior $p(\hat{X}_{0}~|~\hat{Y}_{0}, U_{0})$. The conditional probability $p(\hat{Y}_{0:t}~|~U_{{0:t}-1})$  can be obtained using standard approaches \cite{Gyorfi:2007aa,Brand:1999aa}. Ongoing research focuses on enhancing our understanding of the computational implications in learning $p(\hat{Y}_{0:t}~|~U_{{0:t}-1})$ in real time.
\end{remark}

As we operate both the actual system and the model using the separated control strategy (Fig. \ref{fig:2}), we compute $p(\hat{X}_{t} ~|~ \hat{Y}_{0:t}, U_{0:t-1})$ and learn $p(\hat{Y}_{0:t}~|~U_{{0:t}-1})$ that allows us to compute the information state $\Pi_{t}(Y_{0:t}, U_{0:t-1})(X_{t},$ $\hat{X}_{t})$ (the conditional probability ($p(X_{t} ~|~ Y_{0:t}, U_{0:t-1})$ is known a priori from the model). Next, we show that when the information state $\Pi_{t}(Y_{0:t}, U_{0:t-1})(X_{t},$ $\hat{X}_{t})$ becomes known, then the separated control strategy is optimal for the actual system.

\begin{theorem} \label{theo:CPSmodel}
	Let $\textbf{g}\in\mathcal{G}^s$ be an optimal separated control strategy parameterized with respect to $\hat{x}_{0:T}$, derived offline using the system's model, that minimizes the following cost function,
	\begin{align}
		J(\textbf{g};\hat{x}_{0:T})&\coloneqq \mathbb{E}^{\textbf{g}}\Bigg[\sum_{t=0}^{T-1}\Big[ c_t(X_t,U_t)+ \beta\cdot|X_{t+1} - \hat{X}_{t+1}|^2\Big] \nonumber\\&+ c_T(X_T)  \Bigg],
		\label{theo:CPSmodelgo}	
	\end{align}	
	in Problem 2. Then, if $p(X_{t},\hat{X}_{t}~|~Y_{0:t}, U_{0:t-1})$ becomes known, then $\textbf{g}$ is also optimal for Problem 1, 
	\begin{align}\label{eq:CPScost}
		\hat{J}(\textbf{g})=\mathbb{E}^{\textbf{g}}\left[\sum_{t=0}^{T-1} c_t(\hat{X}_t, U_t)+c_T(\hat{X}_T)\right].
	\end{align}
\end{theorem}
\begin{proof}
  Suppose that the minimum value of the cost function $c_T(X_T)$ occurs at $X_T= x_T\in\mathbb{R}^n$. Hence, 
    \begin{align}
    c_T(X_T=x_T) = c_T(\hat{X}_T=x_T).
    \end{align}
    
  Suppose that the minimum value of the cost function $c_t(X_t,U_t)$ at $t =0,\ldots,T-1$ occurs at $X_t= x_t\in\mathbb{R}^n$ and corresponds to the optimal control $U_t=u_t^*$.   
  Then, the minimum value in the one-time-step cost in \eqref{theo:CPSmodelgo} at  $t= 0,\ldots,T-1$ is when the expected value of the cost function is $c_t(x_t,u_t^*)$ and $\mathbb{E}[|X_{t+1} - \hat{X}_{t+1}|^2]=0,$ hence
	\begin{align}
		\min_{u_t}\mathbb{E}^{\textbf{g}}\Big[ c_t(X_t,u_t)+ \beta\cdot|X_{t+1} - \hat{X}_{t+1}|^2\Big]
            =c_t(x_t,u_t^*).
		\label{onetime}	
  	\end{align}	

    Since at each time $t= 0,\ldots, T-1$, the separated control strategy $\textbf{g}\in\mathcal{G}^s$ yields a control input $u_t'  = g_t\big(p(x_{t},\hat{x}_{t}~|~y_{0:t}, u_{0:t-1})\big)$ such that
	\begin{align}\label{argmin}
u_t' &=\arg\min_{u_t} \mathbb{E}\Big[c_t(X_t,u_t)+ \beta\cdot|X_{t+1} - \hat{X}_{t+1}|^2\Big]\nonumber\\
&= \arg\min_{u_t} c_t(x_t,u_t^*),
 	\end{align}	
   this implies that $u_t'=u_t^*$. 
   
   By summing up all minimum expected values of the cost function $c_t(\cdot, \cdot)$ at each $t=0,\ldots,T-1$ and $c_T(\cdot)$ at $t=T$ corresponding to $\textbf{g}\in\mathcal{G}^s$, we obtain \eqref{eq:CPScost}.
\end{proof}

\section{Illustrative Example} \label{sec:4}

In this section, we present a simple example of deriving the optimal control strategy for a  linear system by separating the learning and control tasks. The purpose of the example is to demonstrate in simple steps the proposed framework.
The primitive random variables, i.e., the initial state, $X_0$, and disturbance, $W_0,$ of the system,  are Gaussian random variables with zero mean, variance $1$, and covariance $0.5$. The state of the actual system is denoted  by $\hat{X}_t, ~t=0, 1,2.$ The evolution of the system is described by the following equations

\begin{align}\label{eq:example1}
	\hat{X}_0 &= X_0, \nonumber\\
	\hat{X}_1 &= \hat{X}_0 +U_0 + W_0 = X_0 +U_0 + W_0, \nonumber\\
	\hat{X}_2 &= \hat{X}_1 + U_1.
\end{align}

We assume that we have a complete observation of the state, i.e.,  
\begin{align}
	\hat{Y}_t = \hat{X}_{t}, \quad t=0, 1,2.
\end{align}

A control strategy $\textbf{g}=\{g_t;~ t=0,1\}$, $\textbf{g}\in\mathcal{G}$, where  $g_t$ is the control law, yields the control action $U_t,$ $t = 0,1,$ of the system, i.e.,
\begin{align}\label{eq:controlexample}
	U_0 &=g_0(\hat{Y}_{0})= g_0(\hat{X}_{0})=g_0(X_{0}),\\
	U_1 &=g_1(\hat{Y}_{0:1}, U_{0})=g_1(\hat{X}_{0:1}, U_{0})=g_1(X_0, \hat{X}_{1}, U_{0}).
\end{align}

We seek to derive the optimal control strategy $\textbf{g}^*\in\mathcal{G}$ of the system represented in \eqref{eq:example1} which minimizes the following  expected cost:
\begin{align}
	J(\textbf{g})&=\min_{u_0\in\mathcal{U}_0, u_1\in\mathcal{U}_1}\frac{1}{2}\mathbb{E}^{\textbf{g}}\left[(\hat{X}_2)^2 + (U_1)^2\right].
\end{align}

We pretend that the dynamics of the system in \eqref{eq:example1} (the actual system) are not known, but  we have available the following model that can be used to  obtain $\textbf{g}\in\mathcal{G}$:
\begin{align}\label{eq:example11}
	X_0 &= X_0,\nonumber\\
	X_1 &= 3 X_0 + 2 U_0 + 2 W_0, \nonumber\\
	X_2 &= 3 X_1 + 3 U_1,
\end{align}
with	
\begin{align}
	Y_t = X_{t},~ t= 0,1,2.
\end{align}
From \eqref{eq:example1} and \eqref{eq:example11}, we note that there exists a  discrepancy between the actual system and the model that is available.

\subsection{Optimal Control Strategy}
First, we obtain the optimal control strategy $\textbf{g}^*\in\mathcal{G}$ of the actual system using  \eqref{eq:example1}.

The cost for the actual system \eqref{eq:example1} is
\begin{align}\label{eq:example6}
	J(\textbf{g})&=\min_{u_0\in\mathcal{U}_0, u_1\in\mathcal{U}_1}\frac{1}{2}\mathbb{E}^{\textbf{g}}\left[(\hat{X}_2)^2 + (U_1)^2\right]\nonumber\\
	&=\min_{u_0\in\mathcal{U}_0, u_1\in\mathcal{U}_1}\frac{1}{2}\mathbb{E}^{\textbf{g}}\left[(\hat{X}_1 + U_1)^2 + (U_1)^2 \right]\nonumber\\
	&=\min_{u_0\in\mathcal{U}_0, u_1\in\mathcal{U}_1}\frac{1}{2}\mathbb{E}^{\textbf{g}}\left[(X_0 + U_0 +W_0 + U_1)^2 + (U_1)^2 \right].
\end{align}

If the dynamics of the actual system given in \eqref{eq:example1} were known, then we could use \eqref{eq:example6} to derive the optimal control strategy $\textbf{g}^*\in\mathcal{G}$.
Since the primitive random variables are Gaussian with zero mean, variance $1$, and covariance $0.5,$ we can use the linear least-squares estimator to compute the unique optimal solution of  \eqref{eq:example6}, which is

\begin{align}\label{eq:example8}
	U_0 = -\frac{1}{2}X_0,\quad U_1=-\frac{1}{4}X_0 - \frac{1}{2}W_0.
\end{align}

\subsection{Solution  Through Separation Between Learning and Control}
In this section, we consider that the dynamics of the actual system \eqref{eq:example1} are not known, but we have the model \eqref{eq:example11} of the system available. Using this model,  we can obtain the optimal control strategy by applying the framework presented in Section \ref{sec:3}. More specifically, we use  \eqref{eq:example11}  and seek to derive the separated control strategy $\textbf{g}\in\mathcal{G}^s$, $\textbf{g}=\{g_t;~t = 0,1\}$, where the control laws are  $g_0\big(\mathbb{P}(X_0,\hat{X}_0~|~Y_{0})\big)$ and $g_1\big(\mathbb{P}(X_1,\hat{X}_1~|~Y_{0}, Y_1, U_{0})\big)$, that minimizes the following  cost (see Theorem \ref{theo:CPSmodel}),
\begin{align}\label{eq:cost1}
	&J(\textbf{g};\hat{x}_{0:2})\nonumber\\
	&=\min_{u_0\in\mathcal{U}_0, u_1\in\mathcal{U}_1}\frac{1}{2}\mathbb{E}^{\textbf{g}}\big[(X_2)^2 + (U_1)^2 \nonumber\\
	&+\beta (X_1-\hat{X}_1)^2 + \beta (X_2 - \hat{X}_2)^2)~|~X_0, X_1, U_0\big].
\end{align}
From \eqref{eq:example11} and taking $\beta=1$, \eqref{eq:cost1} becomes

\begin{align}
	&J(\textbf{g};\hat{x}_{0:2})\nonumber\\
	&=\min_{u_0\in\mathcal{U}_0, u_1\in\mathcal{U}_1}\frac{1}{2}\mathbb{E}^{\textbf{g}}\Big[(3 X_1 + 3 U_1)^2 + (U_1)^2\nonumber
\end{align}
\begin{align}\label{eq:cost2}
	& + (X_1-\hat{X}_1)^2 + (X_2 - \hat{X}_2)^2)~|~X_0, X_1, U_0\Big]\nonumber\\
	&=\min_{u_0\in\mathcal{U}_0, u_1\in\mathcal{U}_1}\frac{1}{2}\mathbb{E}^{\textbf{g}}\Big[\big(3 (3 X_0 + 2 U_0 + 2 W_0) + 3 U_1\big)^2 \nonumber\\
	&+ (U_1)^2 + (X_1-\hat{X}_1)^2 + (X_2 - \hat{X}_2)^2)~|~X_0, X_1, U_0\Big].
\end{align}
The cost in \eqref{eq:cost2} becomes equal to the  original cost in \eqref{eq:example6}, if the control action $U_0$ and $U_1$ make the last two terms equal to zero, i.e.,
\begin{align}
	&\mathbb{E}^{\textbf{g}}[X_1 - \hat{X}_1] = \mathbb{E}^{\textbf{g}}[3 X_0 + 2 U_0 + 2 W_0-\hat{X}_1~|~X_0] = 0, \label{eq:cost3a}\\
	& \mathbb{E}^{\textbf{g}}[X_2 - \hat{X}_2] = \mathbb{E}^{\textbf{g}}[3 X_1 + 3 U_1 - \hat{X}_2~|~X_0, X_1, U_0] = 0. \label{eq:cost3b}
\end{align}

From \eqref{eq:cost3a}, it follows that
\begin{align}
	& \mathbb{E}^{\textbf{g}}[U_0] = \mathbb{E}^{\textbf{g}}\Big[ \frac{\hat{X}_1 - 3 X_0 - 2 W_0}{2}~|~X_0\Big] \nonumber\\
	&= g_0\big(p(X_0,\hat{X}_0~|~X_0)\big). \label{eq:cost4a}
\end{align}

Similarly, from \eqref{eq:cost3b}, it follows that
\begin{align}
	\mathbb{E}^{\textbf{g}}[U_1] &= \mathbb{E}^{\textbf{g}}\Big[\frac{\hat{X}_2 - 3 X_1}{3}~|~X_0, X_1, U_0\Big] \nonumber\\
	&= \mathbb{E}^{\textbf{g}}\Big[\frac{\hat{X}_2 - 3 (3X_0 +2 U_0+2W_0)}{3}~|~X_0, X_1, U_0\Big] \nonumber\\
	&= \mathbb{E}^{\textbf{g}}\Big[\frac{\hat{X}_2 - 9 X_0 - 6 U_0 - 6 W_0}{3}~|~X_0, X_1, U_0\Big] \nonumber\\
	&= g_1 \big(p(X_1,\hat{X}_1~|~X_0, X_1, U_0)\big). \label{eq:cost4b}
\end{align}

Thus, $U_0$ and $U_1$ in \eqref{eq:cost4a} and \eqref{eq:cost4b}, respectively, are parameterized with respect to the expected values of the state of the actual system, i.e., $\hat{x}_0= x_0$, $\hat{x}_1$ and $\hat{x}_2$, and  make the last two terms in \eqref{eq:cost2} vanish. 

Next, we use the control actions $U_0$ and $U_1$ derived by the separated control strategies  $g_0\big(p(X_0,\hat{X}_0~|~X_0)\big)$ and $g_1 \big(p(X_1,\hat{X}_1~|~X_0, X_1, U_0)\big)$ in \eqref{eq:cost4a} and \eqref{eq:cost4b}, respectively, to control both the actual linear system \eqref{eq:example1} and the model \eqref{eq:example11} (see Fig. \ref{fig:2}). As we operate both systems,  we compute  the information states $p(X_{0},\hat{X}_{0} ~|~ X_{0})$ and $p(X_1, \hat{X}_{1} ~|~ X_{0}, X_{1}, U_{0})$. However, from Proposition \ref{theo:CPSstate}, we know that to compute $p(X_{0},\hat{X}_{0} ~|~ X_{0})$ and $p(X_1, \hat{X}_{1} ~|~ X_{0}, X_{1}, U_{0}),$ it is sufficient to compute the conditional probabilities $p(X_{0} ~|~ X_{0})$, $p(X_{1} ~|~ X_{0},X_{1}, U_{0})$, and $p(\hat{X}_{0} ~|~ \hat{X}_{0}, \hat{X}_{1}, U_{0})$, and to learn $p(\hat{X}_{0},\hat{X}_{1}~|~U_{0},U_{1})$. Once we compute these conditional probabilities, the expected values of $U_0$ and $U_1$ in \eqref{eq:cost4a} and \eqref{eq:cost4b} become known.

By substituting \eqref{eq:cost4a} in \eqref{eq:cost2}, we obtain
\begin{align}\label{eq:cost5}
	J(\textbf{g};\hat{x}_{0:2}) & = \min_{u_0\in\mathcal{U}_0, u_1\in\mathcal{U}_1}\frac{1}{2}\mathbb{E}^{\textbf{g}}\Big[\big(3 (3 X_0 + 2 \frac{\hat{X}_1 - 3 X_0 - 2W_0}{2} \nonumber\\
	&+ 2 W_0) + 3 U_1\big)^2+ (U_1)^2  ~|~X_0, X_1, U_0\Big] \nonumber\\
	& =  \min_{u_0\in\mathcal{U}_0, u_1\in\mathcal{U}_1}\frac{1}{2}\mathbb{E}^{\textbf{g}}\Big[\big(3 \hat{X}_1 + 3 U_1\big)^2 + (U_1)^2  \nonumber\\
	&~|~X_0, X_1, U_0\Big]\nonumber\\
	& =  \min_{u_0\in\mathcal{U}_0, u_1\in\mathcal{U}_1}\frac{1}{2}\mathbb{E}^{\textbf{g}}\Big[\big(3 (X_0 +U_0 + W_0) + 3 U_1\big)^2 \nonumber\\
	&+ (U_1)^2  ~|~X_0, X_1, U_0\Big].
\end{align}

However,  at $t=0$, we do not consider $U_1$ and $X_1$. Thus, the last equation becomes

\begin{align}\label{eq:cost7}
	 \min_{u_0\in\mathcal{U}_0}\frac{1}{2}\mathbb{E}^{\textbf{g}}\Big[\big(3 (X_0 +U_0 + W_0) \big)^2 ~|~X_0\Big].
\end{align}
The optimization problem above is to choose for each value $X_0$ the best estimate, in a mean squared error sense, of $(X_0 +U_0 + W_0),$ which yields $U_0 = -\frac{1}{2}X_0$ which is the same solution as in \eqref{eq:example8}. By substituting \eqref{eq:cost4b} into  the model \eqref{eq:example11}, we obtain
\begin{align}\label{eq:cost8}
	X_2 &= 3 X_1 + 3 \frac{\hat{X}_2 - 9 X_0 - 6 U_0 - 6 W_0}{3} \nonumber\\
	&= 3 (3 X_0 + 2 U_0 + 2 W_0) + \hat{X}_2 - 9 X_0 - 6 U_0 - 6 W_0\nonumber\\
	&= \hat{X}_2 ,
\end{align}
hence the expected total cost $J(\textbf{g};\hat{x}_{0:2})$ in \eqref{eq:cost1} becomes

\begin{align}\label{eq:cost9}
	J(\textbf{g};\hat{x}_{0:2}) &= \min_{u_0\in\mathcal{U}_0, u_1\in\mathcal{U}_1}\frac{1}{2}\mathbb{E}^{\textbf{g}}\Big[(\hat{X}_2)^2 + (U_1)^2 )\Big] \nonumber\\
	&= \min_{u_0\in\mathcal{U}_0, u_1\in\mathcal{U}_1}\frac{1}{2}\mathbb{E}^{\textbf{g}}\Big[(\hat{X}_1 + U_1)^2 + (U_1)^2 )\Big].
\end{align}

The minimum in \eqref{eq:cost9} at time $t=1$ can be found by taking the partial derivative with respect to $U_1$ which yields

\begin{align}\label{eq:cost10}
	&\mathbb{E}^{\textbf{g}}\Big[\big( \hat{X}_1 + U_1 + U_1 \big)  \Big] = \mathbb{E}^{\textbf{g}}\Big[\big( X_0 + U_0 +W_0 + U_1 + U_1 \big)  \Big] \nonumber\\
 &= 0
\end{align}
that results in  the same solution $U_1=-\frac{1}{4}X_0 - \frac{1}{2}W_0$ as in \eqref{eq:example8}.

\section{Concluding Remarks}\label{sec:5}
In this paper, we presented a theoretical framework that provides a data-driven  approach for linear systems at the intersection of learning and control. 
The framework  separates the control and learning tasks which allows us to combine offline model-based control with online learning approaches and thus circumvent current challenges in deriving optimal control strategies. 
One feature that distinguishes the framework presented here from other learning-based or combined learning and control approaches reported in the literature is that the large volume of data added to the system is compressed to sufficient
statistics, without loss of optimality, that takes values in a time-invariant space. Hence, as the volume of data added to the systems increases,  the domain of the control strategies does not increase with time. Ongoing research investigates the computational implications of learning the information state.
In our exposition, we restricted attention to centralized control systems. A potential future research direction includes expanding the framework to decentralized systems \cite{Malikopoulos2021}.

\bibliographystyle{IEEEtran.bst} 
\bibliography{reference/TAC_learn_Andreas,reference/TAC_Ref_Andreas,reference/TAC_Ref_IDS,reference/TAC_Ref_structure,reference/TAC2_learn}

\end{document}